\documentclass[11pt]{amsart}
\usepackage{pgf,tikz,pgfplots,color}
\pgfplotsset{compat=1.15}
\usepackage{mathrsfs} 
\usetikzlibrary{arrows}
\usepackage{fancyhdr}
\usepackage{lastpage} 

\usepackage[margin=1in]{geometry}
\usepackage{geometry}
\usepackage[toc,page]{appendix}
\usepackage[utf8]{inputenc}
\usepackage{hyperref}
\hypersetup{
	colorlinks=true,
	linkcolor=blue,
	citecolor=brown,
	filecolor=magenta, 
	urlcolor=blue,
}
\usepackage[notcite,notref]{}  
\usepackage{amsmath}
\usepackage{amsfonts}
\usepackage{amssymb}
\usepackage{amsthm}
\usepackage{setspace}
\usepackage{inputenc}
\usepackage[english]{babel} 
\usepackage{comment}
\usepackage[shortlabels]{enumitem}
\numberwithin{equation}{section}

\makeatletter
\def\@tocline#1#2#3#4#5#6#7{\relax
  \ifnum #1>\c@tocdepth 
  \else
    \par \addpenalty\@secpenalty\addvspace{#2}%
    \begingroup \hyphenpenalty\@M
    \@ifempty{#4}{%
      \@tempdima\csname r@tocindent\number#1\endcsname\relax
    }{%
      \@tempdima#4\relax
    }%
    \parindent\z@ \leftskip#3\relax \advance\leftskip\@tempdima\relax
    \rightskip\@pnumwidth plus4em \parfillskip-\@pnumwidth
    #5\leavevmode\hskip-\@tempdima
      \ifcase #1
       \or\or \hskip 1em \or \hskip 2em \else \hskip 3em \fi%
      #6\nobreak\relax
    \hfill\hbox to\@pnumwidth{\@tocpagenum{#7}}\par
    \nobreak
    \endgroup
  \fi}
\makeatother

\title[]{Sobolev differentiability properties of logarithmic modulus of real analytic functions}           

\author[]{Ziming Shi} 

\author[]{Ruixiang Zhang} 

\address{Ziming Shi (Corresponding author), Department of Mathematics,
   University of California - Irvine, Irvine, CA, 92697}
\email{zimings3@uci.edu}
 
\address{Ruixiang Zhang, Department of Mathematics,
	University of California - Berkeley, Berkeley, CA, 94720}
\email{ruixiang@berkeley.edu} 


\newcommand{\dist}{\operatorname{dist}}

\newcommand{\supp}{\operatorname{supp}}

\newcommand{\loc}{\mathrm{loc}}

\newtheorem{thm}{Theorem}[section]
\newtheorem{cor}[thm]{Corollary} 
\newtheorem{prop}[thm]{Proposition}
\newtheorem{lemma}[thm]{Lemma}

\theoremstyle{definition}
\newtheorem{defn}[thm]{Definition}
\newtheorem{exmp}[thm]{Example}
\newtheorem{ques}[thm]{Question}

\theoremstyle{remark}
\newtheorem{rem}[thm]{Remark}
\newtheorem*{clm}{Claim}
\newtheorem*{ack}{Acknowledgment}

\renewcommand{\th}[1]{\begin{thm}\label{#1}}
	\renewcommand{\eth}{\end{thm}}
\newcommand{\co}[1]{\begin{cor}\label{#1}}
	\newcommand{\eco}{\end{cor}}
\newcommand{\pr}[1]{\begin{prop}\label{#1}}
	\newcommand{\epr}{\end{prop}}

\newcommand{\df}[1]{\begin{defn}\label{#1}}
	\newcommand{\edf}{\end{defn}}
\newcommand{\ex}[1]{\begin{exmp}\label{#1}} 
	\newcommand{\eex}{\end{exmp}}
\newcommand{\qu}[1]{\begin{ques}\label{#1}}
	\newcommand{\equ}{\end{ques}}  
\newcommand{\mk}{\begin{rem}}
	\newcommand{\emk}{\end{rem}}
\newcommand{\cl}{\begin{clm}}
	\newcommand{\ecl}{\end{clm}} 
\newcommand{\ac}{\begin{ack}}
	\newcommand{\eac}{\end{ack}} 

\newcommand{\ga}{\begin{gather}}
\newcommand{\ega}{\end{gather}}
\newcommand{\gan}{\begin{gather*}}
\newcommand{\egan}{\end{gather*}}
\newcommand{\al}{\begin{gngn}}
	\newcommand{\eal}{\end{align}}
\newcommand{\aln}{\begin{align*}}
\newcommand{\ealn}{\end{align*}}
\newcommand{\eq}[1]{\begin{equation}\label{#1}}
\newcommand{\eeq}{\end{equation}}

\newcommand{\pa}{\partial{}}
\newcommand{\na}{\nabla}

\newcommand{\sm}{\setminus}
\newcommand{\seq}{\subseteq}

\newcommand{\B}{\mathbb{B}} 

\newcommand{\R}{\mathbb{R}} 

\newcommand{\N}{\mathbb{N}}

\newcommand{\V}{\mathcal{V}}
\newcommand{\U}{\mathcal{U}} 
\newcommand{\C}{\mathcal{C}} 

\newcommand{\mc}{\mathcal}

\newcommand{\tit}{\textit} 
\newcommand{\0}{\mathbf{0}}

\newcommand{\ov}{\overline}

\newcommand{\wti}{\widetilde}

\newcommand{\codim}{\operatorname{codim}}

\newcommand{\dbar}{\overline\partial}

\newcommand{\sign}{\operatorname{sign}}

\newcommand{\all}{\alpha}

\newcommand{\del}{\delta}
\newcommand{\Del}{\Delta}
\newcommand{\var}{\varphi}

\newcommand{\ve}{\varepsilon}
\newcommand{\om}{\omega}
\newcommand{\Om}{\Omega}
\newcommand{\thh}{\theta}
 
\newcommand{\La}{\Lambda}

\newcommand{\gm}{\gamma}
\newcommand{\Gm}{\Gamma}

\newcommand{\si}{\sigma}

\newcommand{\re}[1]{(\ref{#1})}

\newcommand{\rl}[1]{Lemma~\ref{#1}}
\newcommand{\rc}[1]{Corollary~\ref{#1}}
\newcommand{\rp}[1]{Proposition~\ref{#1}}
\newcommand{\rt}[1]{Theorem~\ref{#1}}
\newcommand{\rd}[1]{Definition~\ref{#1}}

\newcommand{\nn}{\nonumber}
\newcommand{\nid}{\noindent}

\newcounter{pp}
\newcommand{\bpp}{\begin{list}{$\hspace{-1em}\alph{pp})$}{\usecounter{pp}}}
	\newcommand{\epp}{\end{list}}

\newcounter{ppp}
\newcommand{\bppp}{\begin{list}{$\hspace{-1em}(\roman{ppp})$}{\usecounter{ppp}}}
	\newcommand{\eppp}{\end{list}}

\newcommand{\Ac}{\mathcal{A}}

\newcommand{\Cb}{\mathbb{C}}

\newcommand{\Fc}{\mathcal{F}}

\newcommand{\Oc}{\mathcal{O}}

\newcommand{\Sc}{\mathcal{S}}

\newcommand{\Uc}{\mathcal{U}}
\newcommand{\Wc}{\mathcal{W}}

\newcommand{\blb}{\bigl[} 
\newcommand{\brb}{\bigr]}

\begin{document}
	\definecolor{rvwvcq}{rgb}{0.08235294117647059,0.396078431372549,0.7529411764705882}

\begin{abstract} 
Let $f$ be the germ of a real analytic function at the origin in $\mathbb{R}^n $ for $n \geq 2$, and suppose the codimension of the zero set of $f$ at $\mathbf{0}$ is at least $2$. We show that $\log |f|$ is $W^{1,1}_{\operatorname{loc}}$ near $\mathbf{0}$. In particular, this implies the differential inequality $|\nabla f |\leq V |f|$ holds with $V \in L^1_{\operatorname{loc}}$.   
\end{abstract} 

\maketitle

	
\section{Introduction}

\subsection{Statement of the main results} 
The main goal of this paper is to prove the following results on the log singularity for real analytic functions. 
\begin{thm} \label{Thm::W11_intro} 
Let $f$ be the germ of a real analytic function at the origin in $\R^n$, $n \geq 2$. Suppose that the codimension of the zero set of $f$ at $\0$ (denoted $\codim_{\0} (Z_f)$) is at least $2$. Then there exists a small neighborhood $U$ of $\0$ such that $\log |f| \in W^{1,1} (U)$.
\end{thm} 
For definitions of dimension and codimension of the zero set of a real analytic function, see \rd{Def::dim} below. 

The following global version is an easy consequence of Theorem \ref{Thm::W11_intro}. 
\begin{cor} \label{Cor::W11_glob_intro}
Let $f$ be a real analytic function in a neighborhood of the closure of a bounded open subset $U$ of $\R^n$, $n \geq 2$. Suppose that the codimension of the zero set of $f$ in $U$ is at least $2$. Then $\log |f| \in W^{1,1}(U)$.  
\end{cor} 
\rc{Cor::W11_glob_intro} follows from \rt{Thm::W11_intro} by a simple local-to-global argument. 

We say that $g$ is in $W^{1,1}_\loc(\0)$ (or $L^p_{\loc}(\0)$) at $\0$ if there exists a neighborhood $U$ of $\0$ such that $g \in W^{1,1}(U)$ (or $L^p(U)$), and we say that $g$ has an isolated zero at $\0$ if there exists a neighborhood $U$ of $\0$ such that $Z_f \cap U = \{ \0 \}$, where $Z_f$ denotes the zero set of $f$. As a special case to \rt{Thm::W11_intro}, we conclude that $\log |f| \in W^{1,1}_\loc(\0)$ if $f$ has an isolated zero at the origin. 

\rt{Thm::W11_intro} (and accordingly \rc{Cor::W11_glob_intro}) is sharp up to both the integrability exponent and the codimension of the zero set. Indeed, take the function $f(x,y) = xy$ defined on $\R^2$, where the zero set is the union of $x$ and $y$ axis and has codimension 1. Let $U$ be any neighborhood of the origin, we have
\[
  \int_U \left| \na \log |f| \right|^p \, dV = \int_U \left| \frac{\na f}{f} \right|^p \, dV \approx 
  \int_U \left| \frac{x}{xy} \right|^p \, dV + \int_U \left| \frac{y}{xy} \right|^p \, dV
  = \int_U \frac{1}{|y|^p} \, dV +  \frac{1}{|x|^p} \, dV
\]
which is finite if and only if $p<1$. 
On the other hand, for any $\ve>0$ and $2 \leq n-d \leq n$, there exists a polynomial $f$ with $\codim_\0 Z_f = n-d$, and  $\frac{|\na f|}{f} \notin L^{1+\ve}_\loc (\0)$; see Example \ref{Ex::sharp}. We also show that (\rp{Prop::1-dim}) in one dimension, the only continuous function $f:(-1,1) \to \R$ in $W^{1,p}(-1,1)$ satisfying $f(0) = 0$ and $ \left| \frac{d}{dx} \log |f(x)| \right| = \frac{|f'(x) |}{|f(x)|} \in L^1(-1,1)$ is the zero function. 
 
Given any germ of a real analytic function $f$ at the origin with $f(\0) =0$ and $f$ not identically $0$, we can show that $\log |f| \in L^p_{\loc}(\0)$ for any $0 <p< \infty$; this is an easy consequence of the Weierstrass preparation theorem. The main difficulty lies in the derivative estimate: 
\[
 \int_{U \sm Z_f} |\na \log |f|| \, dV = \int_{U \sm Z_f} \frac{|\na f|}{|f|} \, dV < \infty, 
\]
for which the Weierstrass preparation theorem is no longer useful. We can think of the $L^1$ integrability as a result of the cancellation between the zeros of the function and its gradient near the singular points where $\na f = 0$. It should be noted that in many special cases one can do a lot better. Take the simple example $u = |x|^{2k}$ in $\R^n$, with $2 \leq 2k \leq n$, then $|\na u (x)| \approx |x|^{2k-1}$, and  $| \na \log |u(x)| | = \frac{|\na u(x) |}{ |u(x)| } \approx \frac{1}{|x|} \in L^p_\loc$, for any $p<n$. 
In fact, we prove the following result which provides a sharp upper bound for the integrability exponent. 

\begin{thm} \label{Thm::blow-up_intro}  
 Let $f$ be the germ of a real analytic function at the origin in $\R^n$, $n\geq 2$, with $f(\0) = 0$. Suppose $\codim_\0 Z_f = n-d \geq 1$ (i.e. $f$ is not identically $0$). Then for each (sufficiently small) neighborhood $\U$ of $\0$, 
 \[
   \int_{\U \sm Z_f} \left| \frac{\na f}{f} \right|^{n-d} \, dV = \infty. 
 \] 
  Furthermore, the integrability exponent $n-d$ is sharp in the sense that there exists a polynomial $f$ such that $ \codim_\0 Z_f = n-d$ and $|\na f|/f \in L^{p}_\loc(\0)$, for any $p<n-d$ (See Example \ref{Ex::sharp}.) 
\end{thm}
We note that in \cite{Pan22}, it was proved that if $f$ is a non-constant locally Lipschitz function $f$ on an open set $\Om$ in $\R^n$ with $Z_f \cap \Om \neq \emptyset$, then $\int_{\Om \sm Z_f} |\na \log |f||^n\, dV = \infty$.  

To illustrate the above results, take the function $h = x^2y^2 + z^2$ in $\R^3$. Its zero set is the union of the $x$ and $y$ axes and therefore has codimension $2$. Our theorems imply that $\int_U |\na h| / |h| \, dV <\infty$ and $\int_U (|\na h| / |h|)^2 \, dV = \infty$ for any bounded open set $U$ containing $\0$ in $\R^3$. The first statement is not obvious. To see that $|\na h| / h$ is not $L^2$ integrable on $U$, we can use a simple direct argument. Assume that $U_\ve = \{ |x| \approx c, |y|, |z| \approx \ve \} \subset U$ for some 
$c>0$. Then $ \int_{U_\ve} (|\na h| / |h|)^2 \, dV$ is roughly $\ve^2 \cdot \ve^{-2} \approx 1$. The statement then follows by summing over $\ve$ in dyadic intervals approaching to $0$. 

The proof of \rt{Thm::W11_intro} relies on the ``finiteness" property of analytic functions and their zero sets (which are analytic sets). A particularly useful fact is that analytic sets can be decomposed locally into finitely many connected analytic manifolds. We shall generalize this phenomenon and adopt the viewpoint that analytic sets are (locally) definable in some o-minimal structure, and accordingly they are ``nice" and have limited complexity, or ``tame" in their geometry and topology. 

Another way to state our result is that if a real analytic germ $f$ satisfies the hypothesis given in \rt{Thm::W11_intro} on the codimension of its zero set, then $f$ satisfies the differential inequality $|\na f| \leq V |f|$, for $V \in L^1_\loc(\0)$.
Conversely, it is well-known that certain differential inequalities imply analytic-like properties. In fact, our original motivation comes from study of the unique continuation properties of the differential inequality
\eq{Lap_diff_ineq} 
  \left| \Del u \right| \leq A \left| u \right| + B| \na u|, \quad A \in L^p_{\loc}(\0), \; B \in L^q_\loc(\0), \; 0<p,q<\infty.   
\eeq
We say that the differential inequality \re{Lap_diff_ineq} satisfies the unique continuation property (UCP) at $\0$, if every solution that vanishes in a neighborhood of a point $\0$ vanishes identically. 
Much work have been done to find the optimal (minimal) values for $p$ and $q$ such that \re{Lap_diff_ineq} satisfies the UCP. See for example \cite{Wol95} for an exposition on this problem. 
We mention yet another related result by Gong-Rosay  \cite{G-R07}, which shows that if $f:\Om \to \Cb$ is a continuous map on some open set $\Om$ in $\Cb^n$ and $|\dbar f | \leq K|f|$ on $\Om \sm Z_f$ for some constant $K$, then $Z_f$ is a (complex) analytic set.

\subsection{An application to local invariants}
 First, we recall the Łojasiewicz gradient inequality. Let $f$ be the germ of a real analytic function at $\0$ such that $f(\0)=0$. There exists some $\beta \in (0,1)$ and a small neighborhood $\V$ of $\0$ such that 
\begin{equation} \label{L_ineq}  
  | \na f (x)| \geq c_\beta |f(x)|^\beta, \quad \text{for all $x \in \V$.}
\end{equation}
Here we may assume $\na f(\0)= 0$ as otherwise the inequality is trivial. 
The \emph{Łojasiewicz exponent of $f$ at $\0$}, denoted by $\beta_0$, is defined to be the infimum of all $\beta$ satisfying \re{L_ineq}. 

For any real analytic function $f$ with $f(\0)=0$, we define the \emph{singularity exponent of $f$ at $\0$}, denoted by $\all_0$, as the supremum of all $\all >0$ such that there exists a small neighborhood $\V$ of $\0$ with $\int_{\V} |f|^{-\all} < \infty$. 
By \rt{Thm::W11_intro}, and inequality \re{L_ineq}, if $\codim_\0 Z_f \geq 2$, then there exists a neighborhood $\U$ of the origin such that
\[
  \infty > \int_{\U} \left| \frac{\na f}{ f } \right| \, dV 
  \geq c_\beta \int_{\U} \frac{1}{|f|^{1-\beta}} \, dV
\] 
for any $\beta > \beta_0$. This implies that $(1-\beta) < \all_0$. 
Letting $\beta \to \beta_0$, we get $1- \beta_0 \leq \all_0$, or  
\[
\all_0 + \beta_0 \geq 1. 
\]  
We summarize the result in the following corollary:
\begin{cor}
  Let $f$ be the germ of a real analytic function at the origin in $\R^n$ with $f(\0) = 0$. Suppose $\codim_\0 Z_f \geq 2$. Let $\all_0$ and $\beta_0$ be the singularity exponent and the Łojasiewicz exponent of $f$ at $\0$, respectively. Then 
	\[
	\all_0 + \beta_0 \geq 1. 
	\] 
\end{cor} 

We point out that all of our results fail rather spectacularly if $f$ is only assumed to be $C^\infty$. Take $u_\ve (x) = e^{-|x|^{-\ve}}$ defined in $\R^n$, with $\ve>0$. Then $ \log u_\ve (x) = - |x|^{-\ve}$. It follows that for any $p >0$, $\log u_\ve \notin L^p_\loc(\0) $ whenever $\ve \geq n/p$.  
 
We use $x \lesssim y$ to mean that $x \leq Cy$ where $C$ is a constant independent of $x,y$, and we write $x \approx y$ if $x \lesssim y$ and $y \lesssim x$. For an open subset $\Om$ of $\R^n$, we denote by $C^\infty_c(\Om)$ the space of $C^\infty$ function with compact support in $\Om$. We denote the zero set of $f$ by $Z_f$, and we write the volume element in $\R^n$ as $dV$.   

\begin{ack}
The authors would like to thank Yifei Pan for many inspiring and helpful discussions. They also thank the referee for many valuable comments. 
\end{ack}

\section{Preliminaries} 
In this section we review some definitions and concepts used in later proofs. 

Let $\mc M$ be a connected $n$-dimensional analytic manifold and $\Uc$ be an open subset of $\mc M$. We denote by $ \mathcal{O}^\R_{\Uc}$ the ring of real analytic functions from $\Uc$ to $\R$. If $p \in \mc M$, we denote by $\Oc^\R_{\mc M,p}$ the set of real analytic germs at $p$. 
We denote the set of polynomials in $\R^n$ by $\R[x] = \R[x_1,\dots, x_n]$.
\begin{defn} 
 Let $A \subset \Oc_{\mc M}^\R$. We define the \emph{vanishing locus} of $A$, $V(A)$, to be the set of points: 
 \[
 V(A):= \{ x \in \mc M: f(x) = 0, \: \forall \: f \in A \}.
 \] 
\end{defn}
\begin{defn} 
 An \emph{algebraic subset of $\R^n$} is a set of the form $V(\Ac)$, where $\Ac \seq \R[x]$. A subset $X \subset \mc M$ is a \emph{(real) analytic subset} of $\mc M$ if $X$ is closed in $\mc M$ and, for all $x \in X$, there exists an open neighborhood $\mc{W} $ of $x$ in $\mc M$ and a finite collection $f_1, \cdots, f_j \in \mc{O}^{\R}_{\mc{W}}$ such that $\mc{W} \cap X = V (f_1, \cdots, f_j)$. 
\end{defn}
\begin{defn}
Let $X$ be an analytic subset of $\mc M$. A point $p \in X$ is called \emph{smooth, of dimension $d$}, if there exists an open neighborhood $\mc{W}$ of $p$ in $\mc M$ such that $\mc{W} \cap X$ is an analytic sub-manifold of $\mc{W}$ of dimension $d$. In other words, there exists $f_{d+1}, \dots, f_n \in \Oc_{\Wc}^{\R}$ such that $\Wc \cap X = V(f_{d+1}, \dots, f_n)$ and  $\na f_{d+1}(x), \cdots, \na f_n(x) $ are linearly independent at each $x \in \Wc$.  

We denote the set of smooth points of $X$ by $\mathring{X}$, and the set of smooth points of dimension $d$ by $\mathring{X}^{(d)}$. 
\end{defn}

\df{Def::dim} 
The \emph{dimension} (over $\R$), $\dim X$ of an analytic set $X \subset \mc M$ is the largest $m$ such that $\mathring{X}^{(m)}$ is non-empty. The \emph{dimension of $X$ at a point} $p \in X$, denoted by $\dim_p X$, is the largest $d$ such that $p$ is in the closure of $\mathring{X}^{(d)} $. We say that $X$ is \emph{pure-dimensional} if the dimension of $X$ at each point $p \in X$ is independent of $p$. 
The \emph{codimension $\codim(X)$ of an analytic set $X \subset \mc M$} is defined as $n-d$, where $d =\dim(X)$. The \emph{codimension of $X$ at a point $p \in X$} is defined as $n-d_p$, where $d_p := \dim_p X$.  
\edf  

Analytic sets can be partitioned into smooth sets, as the following result (see \cite{B-M88}) shows. 
\begin{prop}[Stratification of analytic sets] 
 Let $X$ be an analytic subset of $M$. Then there exists a collection $\{ \Ac_\all \}_\all$ of subsets $M$ such that 
\begin{itemize}
    \item 
    X is the disjoint union of the $\Ac_\all$; 
    \item 
    Each $A_\all$ is an analytic submanifold of $M$;   
    \item 
    (``Condition of the frontier") If $\Ac_\all \cap \ov {\Ac_\beta} \neq \emptyset$, then $\Ac_\all \seq \ov{\Ac_\beta}$ and $\dim \Ac_\all < \dim \Ac_\beta$;  
    \item
    $ \{ \Ac_{\all} \}_\all$ is locally finite.  
\end{itemize} 
\end{prop}

We can also define the \emph{dimension of an analytic set $X$} by $\dim X = \max_k \dim \Ac_k$. The definition is independent of the stratification: $\dim X = d$ if and only if $X$ contains an open set homeomorphic to an open ball in $\R^d$, but not an open set homeomorphic to an open ball in $\R^n$, $n >d$. It is also clear that the definition agrees with that of the Hausdorff dimension. 

One result we will be using is the {\L}ojasiewicz distance inequality for analytic functions: 
\begin{prop} \label{Prop::L_dist_ineq}
 Let $g$ be a real analytic function in a neighborhood $U$ of the origin in $\R^n$. Then for any compact set $K$ in $U$, there exist constants $c>0, \all >0$ which depend only on $g$, such that 
\[ 
   |g (x) | \geq c \dist (x, Z_g)^{\all},   \quad x \in K. 
\] 
\end{prop}
Analytic sets are in some sense ``finite" objects. The following proposition gives an important property of this nature. 
\begin{prop}[{\cite[Remark 7.3]{B-M88}}] 
 Let $X$ be an analytic subset of $M$. Then the family of connected components of $X$ is locally finite. 
\end{prop}
As an immediate consequence, we obtain the following 
\begin{prop} \label{Prop::mono_anal}  
  Let $f:(a,b) \to \R$ is analytic, where $(a,b)$ is a bounded interval. For every $(a',b')$ such that $a'>a$ and $b' < b$, there exists $N <\infty$, and $a'=a_0 < a_1 < \cdots < a_N = b' $ such that $f$ is either constant or strictly monotone on each subinterval $(a_i, a_{i+1})$. 
\end{prop}
These properties of analytic sets can be better understood using the theory of o-minimal structures, a concept we recall in the next section.  

\subsection{o-minimality and tame geometry}  
For most of the definitions and results in this section we refer the reader to the book \cite{vdD98}. 
\begin{defn}
  A structure on $\R$ is a sequence $\Sc = \{ \Sc_n\}_{n \in \N}$ such that for each $n$: 
  \begin{enumerate}
      \item  
      $\Sc_n$ is a Boolean algebra of subsets of $\R^n$, that is, $\Sc$ is a collection of subsets of $\R^n$, $ \emptyset \in \Sc$, and if $A,B \in \Sc$, then $A \cup B \in \Sc$ and $\R^n \sm A \in \Sc$; 
      \item 
      $A \in \Sc_n$ $\implies$ $A \times \R \in \Sc_{n+1}$ and $\R \times A \in \Sc_{n+1}$; \item 
      The diagonals $\Del_{ij}:= \{ (x_1, \dots,x_n) \in \R^n: x_i=x_j \} \in \Sc_n$, for $1 \leq i < j \leq n$;  
      \item
      $A \in S_{n+1}$ $\implies$ $\pi(A) \in S_n$, where $\pi: \R^{n+1} \to \R^n$ is the usual projection map; 
        \item 
      $\{(x,y) \in \R^2: x<y \} \in \Sc_2$.  
      \item
      $\Sc_3$ contains the graphs of addition and multiplication. 
    \end{enumerate} 
    $\Sc$ is called \emph{o-minimal} if it satisfies the following additional axiom:  
   \begin{quote}
        The sets in $\Sc_1$ consists of exactly the finite unions of intervals and points.  
   \end{quote}
 
\end{defn}
Fix an o-minimal structure $\Sc$. Let $A \subset \R^m$ and $f: A \to \R^n$. We say $A$ is $\Sc$-$\emph{definable}$, or simply \emph{definable} when the underlying structure $\Sc$ is clear, if $A \in \Sc_m$; we say the \emph{map $f$ is definable} if its graph $\Gm(f) \subset \R^{m+n}$ is definable. If $f$ is definable, then the domain $A$ of $f$ and its image $f(A)$ are also definable.  

A structure is usually constructed as follows. Consider a family of functions, denoted by $F$, and take the smallest structure containing the graphs of all the functions in $F$. When the family consists of all constant functions, i.e. $f = c$, $c \in \R$, and also the graphs of addition and multiplication in $\R^3$, we obtain the family of semialgebraic sets, which is an $o$-minimal structure. On the other hand, if $F$ ranges over all restricted analytic functions, we obtain the family of so-called globally subanalytic sets, denoted as $\Sc(\R_{an})$. Here we call $f:\R^n \to \R$ a \emph{restricted analytic function}, if there is an open subset $U$ containing $[-1,1]^n$ in $\R^n$, and an analytic function $g:U \to \R$ such that
\[
  f(x) = \begin{cases}
  g(x), &  \text{for $x \in [-1,1]^n$}; 
   \\ 
  \text{constant}, & \text{otherwise.}
  \end{cases}
\] 

The following result due to Gabrielov is important for our application. 
\begin{prop}[{\cite{Gab68}}]
 $S(\R_{an})$ is an o-minimal structure. 
\end{prop} 

From now on we fix some o-minimal structure and talk about definable (or tame) sets and maps respect to this structure. 
\rp{Prop::mono_anal} generalizes to definable functions of o-minimal structure. 
\begin{prop}[{\cite[Chapter 3, Theorem 1.2]{vdD98}}]  
  Let $f:(a,b) \to \R$ be a definable function on the interval. Then there are points $a_0 = a < a_1 < \cdots < a_N = b$ in $(a,b)$ such that on each subinterval $(a_j, a_{j+1})$, the function is either constant, or strictly monotone. 
\end{prop}  
For our proof we need a parameter version of the above result. We say that a function $f: \R \to \R$ changes monotinicity at a point $x_1$ if for some $x_0<x_1$ and $x_2>x_1$, $f$ changes from being strictly positive (resp. negative) on $(x_0, x_1)$ to strictly negative (resp. positive) on $(x_0,x_2)$. 
\begin{prop}[{\cite[Propostion 2.8]{B-G-Z-K_21}}] \label{Prop::monotonicity}  
  Let $S$ be an o-minimal structure and let $h:\R^m \times \R \to \R$ be a definable function. For $x \in \R^m$, let $N(x)$ denote the number of times the function $h(x,\cdot): \R \to \R$ changes monotonicity. Then $\sup_{x \in \R^m} N(x) <\infty$.  
\end{prop}

\begin{defn}
 We call a set belonging to an o-minimal structure a \emph{tame set}. 
\end{defn}
Next, we need the concept of \emph{cells}, which are non-empty tame sets of particularly simple form. They are defined inductively as follows: 
\begin{enumerate}
    \item 
 the cells in $\R$ are just the points $\{r\}$, and the intervals $(a,b)$; 
 \item 
 Let $\C \subset \R^n$ be a cell; if $f,g: \C \to \R$ are definable continuous functions such that $f < g$ on $\C$, then 
 \[ 
   (f,g) := \{ (x,r) \in \C \times \R: f(x) < r < g(x) \}
 \] 
 is a cell in $\R^{n+1}$. Moreover, given a definable continuous function $f,g: \C \to \R$, the graph of $f$, and the sets 
 \[
  (-\infty,g): = \{ (x,r) \in \C \times \R: r < g(x) \}, \quad (f, \infty):= \{ (x,r) \in \C \times \R: f(x)<r \} 
 \] 
 are cells in $\R^{n+1}$; finally $\C \times \R \subset \R^{n+1}$ is a cell. 
\end{enumerate}
It turns out every tame set can be decomposed into cells. 
\begin{prop}[Cell Decomposition, {\cite[Theorem 2.11, Chapter 3]{vdD98}}] 
  Every tame set $A \subset \R^m$ has a finite partition $A = \C_1 \cup \cdots \cup \C_\ell$ into cells $\C_i$. If $f: A \to \R^n$ is a definable map, this partition of $A$ can be chosen such that all restrictions $f|_{\C_i}$ are continuous. 
\end{prop}
A \emph{$C^k$-cell} in $\R^m$ is defined to be a cell which is also a $C^k$ submanifold of its ambient Cartesian space. 
\begin{prop}[Smooth Cell Decomposition, {\cite[p.~131]{vdD99}}] 
Let $k$ be a positive integer. Every tame set $A \subset \R^m$ admits a finite partition $A = \C_1 \cup \cdots \cup \C_\ell$ into $C^k$-cells. If $f: A \to \R$ is a definable map, then the partition can be chosen such that $f|_{\C_i}$ are $C^k$.    
\end{prop}

The dimension of a cell is defined inductively in an obvious way (For precise definition the reader may refer to \cite[p.~50]{vdD98}), accordingly we define the dimension of a tame set as follows:  
\begin{defn} 
 The \emph{dimension} of a non-empty tame set $A \subset \R^m$ is defined by 
 \[ 
 \dim A:= \max_i \{ \dim(\C_i): \text{ $\cup_i \, \C_i$ is a cell decomposition of $A$} \}, 
 \]
 and $\dim(\emptyset) = -\infty$. The \emph{dimension of $A$ at a point $p \in A$} is defined by 
 \[
   \dim_p A:= \max_i \{ \dim (\C_i), \text{ $p \in \ov{\C_i}$ and $\cup_i \, \C_i$ is a cell decomposition of $A$}. \}   
 \] 
\end{defn}

It can be shown that the above definition does not depend on the choice of the partition, and it also agrees with \rd{Def::dim} for analytic sets, if we consider the latter as a definable set in the o-minimal structure $\R_{an}$. 

Now given a cell decomposition $A = \C_1 \cup \cdots \cup \C_k$, it is easy to see that $\pi (A) = \pi(\C_1) \cup \cdots \cup \pi(\C_k)$ is a cell decomposition of $\pi (A)$, where $\pi: \R^n \to \R^{n-1}$ is the natural projection. Hence we immediately get the following result:
\begin{prop} \label{Prop::dim_proj}
 Let $A$ be a definable set and let $\pi: \R^n \to \R^{n-1}$ be the natural projection map. Then $\dim A \geq \dim \pi(A)$. 
\end{prop}
 
 Let $X \subset \R^n$ be a relatively compact subset, i.e. the closure of $X$ is compact in $\R^n$. For any $\ve>0$, we denote by $M(\ve,X)$ the minimal number of closed balls of radius $\ve$ covering $X$.   

\begin{prop}[{\cite[Corollary 5.7]{Y-C_04}}] 
 Let $A \subset \R^n$ be a tame set of dimension $\ell <n$. Then for any ball $B^n_r$ of radius $r$ in $\R^n$, there exists a constant $C= C(A,n)$ such that  
 \[
  M(\ve, A\cap B_r^n) \leq C \left( \left(\frac{r}{\ve}\right)^l + 1 \right). 
 \] 
\end{prop}
\begin{cor} \label{Cor::vol_tame_set}  
 Let $A \subset \R^n$ be a tame set of dimension $\ell <n$. Denote by $N_\ve(A)$ the $\ve$ neighborhood of $A$. Then
 \[
  |N_\ve(A)| \lesssim \ve^{n-l}, 
 \] 
 where $|\cdot|$ denotes the Lebesgue measure in $\R^n$. 
\end{cor}
\begin{proof}
  Let $\La = \{ B(x_i, \ve) \}_{i=1}^M$ be the set of balls covering $A$, where $M \leq C(n) \left( \left(\frac{r}{\ve}\right)^l + 1 \right)$. Then $N_\ve(A)$ is contained in $\cup_{i=1}^M B(x_i, 2\ve) $, and thus $|N_\ve(A)| \lesssim \ve^n (\ve^{-l} + 1) \lesssim \ve^{n-l}$. 
\end{proof}

\subsection{Removable singularity}  
In our proof of \rt{Thm::W11_intro}, we first show that $ |\na \log |f||  = \frac{|\na f|}{|f|} \in L^1_\loc(\0)$. However, as both $\log |f|$ and $\na \log |f|$ are defined off the zero set of $f$, 
we must show that the derivative $\na \log |f|$ in fact exists in the sense of distribution while crossing the singular set $Z_f$, a phenomenon known as removable singularity. We now state such problem in a slightly more general form. Let $P(x, \Om)$ be a linear partial differential operator on an open set $\Om$ in $\R^n$, and let $A$ be a closed subset of $\Om$. Given a class of distributions on $\Fc (\Om)$, we say (following Harvey-Polking \cite{H-P70}) the set $A$ is \emph{removable} for $\Fc(\Om)$ if each $f \in \Fc(\Om)$ that satisfies $P(x,\Om)f = 0$ in $\Om \sm A$ also satisfies $P(x, \Om) f = 0 $ in $\Om$. Questions can then be asked as to what conditions on the coefficients of $P(x,\Om)$ and on the set $A$ will ensure that $A$ is removable for $\Fc(\Om)$. In the remainder of this section, we will prove two results of this type. 

We let $d(x,E)$ to denote the Euclidean distance from the point $x$ to the set $E \subset \R^n$. We denote the $\ve$-neighborhood of $E$ by $E_{\ve} := \{ x \in \R^n: d(x, E) <\ve \}$. 
The unit ball in $\R^n$ is denoted by $\B$, and the punctured ball $\B \sm \{ \0 \}$ by $\B_\ast$.  
We begin with a result for isolated singularity. 
\pr{Prop::rem_sing_1}  
Let $\all$ be a multi-index with $| \all | = m$, for $1 \leq m<n$. Suppose $f \in L^{\frac{n}{n-m}} (\B)$, $g \in L^1 (\B)$, and
  \[ 
     D^{\all} f =g,  \quad (D^{\all} = \pa_{x_1}^{\all_1} \cdots \pa_{x_n}^{\all_n} ) 
  \]
 holds in the sense of distributions in $\B_{\ast}$. Then $D^{\all} f = g$ holds in the sense of distributions in $\B$. 
\epr
\nid \tit{Proof.} 
Define a smooth function $\var \in C^{\infty}_c (\R^n)$ satisfying 
\[
  \var (x) = 
  \begin{cases} 
  0 & \text{if $|x| < 1$};  \\ 
  1 & \text{if $|x| > 2$}.  
  \end{cases}
\]
Let $\var_n (x) := \var(nx)$. Then $\var_n$ satisifes
\[
 \var_n (x) = 
 \begin{cases} 
 0 & \text{if $|x| < \frac{1}{n} $}; \\ 
 1 & \text{if $|x| > \frac{2}{n}$}.   
 \end{cases} 
\]
Let $\psi \in C^{\infty}_c (\B)$, and set $\psi_n := \psi \var_n$. Since $\psi_n \in C^{\infty}_c (\B_\ast)$, by assumption we have 
\eq{ibp} 
  \int_{\B} f (D^\all \psi_n) = (-1)^{m} \int_{\B} g \, \psi_n. 
\eeq
Since $\psi_n$ converges to $\psi$ point-wise in $\B$ and $g \in L^1 (\B)$, by the Dominated Convergence theorem the integral on right-hand side converges to 
\[
  (-1)^m \int_{\B} g \, \psi. 
\]
For the integral on the left-hand side we have
\begin{align} \label{lhs}
 \int_{\B} f (D^\all \psi_n) 
 &= \int_{\B} f D^\all (\psi \var_n) \\ \nn 
 &= \int_{\B} f (D^\all \psi) \var_n + \sum_{ \substack{|\beta| \geq 1 \\ \beta \leq \all}  }  \int_{\B}   f ( D^{\all - \beta} \psi)  (D^{\beta} \var_n).
\end{align}
Since $f \in L^{\frac{n}{n-m}} (\B) \subset L^1(\B)$, we have
\[ 
  \int_{\B} f ( D^\all \psi) \var_n \xrightarrow{\ve \to 0}  \int_{\B} f (D^\all \psi). 
\]
By the definition of $\var_n$, we see that $D^{\beta} \var_n$ is compactly supported in $\{x: |x| < \frac{2}{n} \}$, and 
\[ 
\left| D^{\beta} \var_n (x) \right| 
= n^{|\beta|} \left| \pa_y^{\beta} \var (nx) \right|
\lesssim n^{|\beta| }.  
\] 
Denote by $B(x,r)$ the ball centered at $x$ with radius $r$. By Hölder's inequality,
\begin{align*}
  \int_{\B} \left|  f ( D^{\all - \beta} \psi)  (D^{\beta} \var_n) \right|  
 &\leq C \left( \int_{|x|< \frac{2}{n}  } |f|^{\frac{n}{n-|\beta|}} \right)^{\frac{n-|\beta|}{n}}  \left( \int_{|x| < \frac{2}{n}} \left| \pa_x^{\beta} \var_n \right|^{\frac{n}{|\beta|}} \right)^{\frac{|\beta|}{n}} \\ 
 &\leq  C \| f \|_{L^{\frac{n}{n- |\beta | }} B(\0, \frac{2}{n}) }  \left( \left( \frac{2}{n} \right)^n n^{|\beta| \cdot \frac{n}{|\beta|}} \right)^{\frac{|\beta|}{n}}   \\
 &\leq C \| f \|_{L^{\frac{n}{n- |\beta | }} B(\0, \frac{2}{n}) } . 
\end{align*} 
Since $|\beta| \leq |\all| = m$, we have $\| f \|_{L^{\frac{n}{n-|\beta|}} \left( B(\0, \frac{2}{n})  \right) } \leq \| f \|_{L^{\frac{n}{n- m }} \left( B(\0, \frac{2}{n})  \right) } $ which converges to $0$ as $n \to \infty$. Putting the results together and letting $n \to \infty$ in \re{ibp} and \re{lhs} we get
\[
  \int_{\B} f (D^\all \psi)  = (-1)^m \int_{\B} g  \psi, \quad \psi \in C^{\infty}_c (\B). 
\]
Hence $D^\all f = g $ in the sense of distribution in $\B$. 
\hfill \qed 

Next, we prove a removable singularity result for an arbitrary set of Hausdorff dimension at most $n-2$.
We will need the following construction of smooth cut-off functions due to Harvey and Polking. For a compact set $K$ in $\R^n$, we let $\La_r (K)$ denote the $r$ dimensional Hausdorff measure of $K$. 
\begin{lemma}[Harvey-Polking Lemma] \label{Lem::HP}
  Let $K$ be a compact subset of $\R^n$. Let $l$ and $p'$ be some positive number such that $n - lp' >0$. For each $\ve >0$, there exists some $\chi_{\ve} \in C^{\infty}_c (\R^n)$ with $\chi_{\ve} \equiv 1 $ in a neighborhood of $K$ and $\supp \chi_{\ve} \subset K_{\ve}$. Furthermore, for $ | \all | <  l$, 
 \eq{flest}
 \left| D^\all \chi_\ve \right|_{L^{p'}} \leq C_{\all,n} \ve^{ l - | \all| } \left( \La_{n-lp'}  (K) + \ve \right)^{\frac{1}{p'}}. 
 \eeq
\end{lemma} 
\begin{proof}
    See \cite[Lemma 3.2]{H-P70}. 
\end{proof}
\begin{prop} \label{Prop::rem_sing_2}  
 Let $\Om$ be an open set in $\R^n$ and let $A \subset \Om$ be a compact set of Hausdorff dimension at most $n-2$. Suppose $f \in C^\infty (\Om \sm A) \cap L^p_\loc(\Om)$ for some $p>2$, and also $\na f \in L^1_\loc(\Om)$. Then $\na f$ is the derivative of $f$ in $\Om$ in the sense of distributions. 
\end{prop}
\begin{proof}
By \rl{Lem::HP}, there exists a family of cut-off functions $ \{ \wti \chi_\ve \}_{\ve>0} \in C^\infty_c(\Om)$ such that $0 \leq \wti \chi_\ve \leq 1 $, $\wti \chi_\ve \equiv 1$ in a neighborhood of $A$, and supp $\wti \chi_\ve \seq A_\ve$. Furthermore, the following estimate holds
\begin{equation} \label{chi_ve_grad_est}   
  \left| \na \wti \chi_\ve \right|_{L^{p'}(\Om)} \leq C_{\all,n} \ve^{ l - 1 } \left( \La_{n-lp'}  (A) + \ve \right)^{\frac{1}{p'}}, \quad n-lp' > 0. 
\end{equation} 
Define $ \chi_\ve := 1 - \wti \chi_\ve$, so that $\chi_\ve \equiv 0$ near $A$ and $\chi_\ve \equiv 1$ in $\Om \sm A_\ve$. Moreover, $\chi_\ve$ converges to $1$ point-wise in $\Om \sm A$ as $\ve \to 0$, and $\chi_\ve$ satisfies estimate \re{chi_ve_grad_est}.     
For $\phi \in C^\infty_c(\Om)$, we can integrate by parts on the domain $\Om \cap \{ \chi_\ve \geq 0 \}$ to get
\begin{equation} \label{ibp_cutoff}  
  \int_{\Om} (\na f) \chi_\ve \phi = - \int_{\Om} f \na (\chi_\ve \phi) = - \int_{\Om} f (\na \chi_\ve) \phi - \int_{\Om} f \chi_\ve \na \phi.  
\end{equation} 
Since $\na f \in L^1_\loc(\Om)$, by the Dominated Convergence theorem, the integral on the left converges to $\int_{\Om} (\na f) \phi$ as $\ve \to 0$. Similarly, the second integral on the right converges to $\int_{\Om} f \na \phi$. By H\"older's inequality with $\frac{1}{p} + \frac{1}{p'} =1$, we get 
\begin{align*}
  \int_{\Om} \left| f (\na \chi_\ve) \phi \right| 
  &\lesssim  \| f \|_{L^p(\Om)} \| \na \chi_\ve \|_{L^{p'}(\Om)}
  \\ &\leq C_{\all,n} \ve^{l- 1} \left( \La_{n-lp'} (A) + \ve \right)^{\frac{1}{p'}} \| f \|_{L^p(\Om)}. 
\end{align*} 
Since $p'<2$ ($p>2$), we can choose $l>1$ such that $lp'<2$, so $n-lp'> n-2$. By the assumption $\dim A \leq n-2$, we have $\La_{n-lp'}(A)=0$, and therefore
\[
 \int_{\Om \sm A} \left| f (\na \chi_\ve) \phi \right| 
 \leq C_{\all,n} \ve^{l-1+\frac{1}{p'}} \| f \|_{L^p(\Om)}. 
\] 
In particular the integral converges to $0$ as $\ve \to 0$. Taking the limit in \re{ibp_cutoff} as $\ve \to 0$, we get
\[
  \int_\Om (\na f) \phi = - \int_\Om f (\na \phi), 
  \quad \phi \in C^\infty_c(\Om). 
\] 
In other words, $\na f$ is the derivative of $f$ in $\Om$ in the sense of distributions. 
\end{proof}
\section{Proof of the Main Theorem}  
In this section we prove the main result of the paper, \rt{Thm::W11_intro}.

\pr{Prop::logLp}  
Let $f$ be the germ of a real analytic function at $\0$ in $\R^n$ and suppose $f(\0)=0$. Then $ \log |f| \in L^{p}_{\loc}(\0) $, for any $0 < p < \infty$. 
\epr
\begin{proof}
By Weierstrass Preparation theorem for real analytic functions (\cite[Theorem 6.1.3]{K-P02}, one can find a real analytic coordinate $x = (x_1, \dots, x_n)$ defined in some cube with length $\del$ centered at $\0$ such that
\eq{Weierstrass_1} 
  f (x) = u (x) \left( x_n^m + a_1(x') x_n^{m-1} + \cdots + a_{n-1}(x') x_n + a_n (x') \right), \quad x'=(x_1, \dots, x_{n-1}).   
\eeq
Here $m$ is the vanishing order of $f$ in the $x_n$ direction, $u(\0) \neq 0$ and $a_j$ are real analytic functions in $n-1$ variables with $a_j (\0') = 0$. We can rewrite \re{Weierstrass_1} as 
\[
  f(x) = u (x) \prod_{j=1}^{m} ( x_n -  \xi_j (x') ),  \quad x \in [-\del, \del]^n, 
\]
where the roots $\xi_j (x')$ are either real or complex numbers. 
Assuming that $|f|< 1$ in $[-\del, \del]^n$, we have 
\begin{align*}
\left| \log |f (x)  | \right| 
&= \log \frac{1}{|f (x) |}  
= \log \frac{1}{|u (x) |} + \sum_{j=1}^m \log \frac{1}{| x_n - \xi_j (x') | }, 
\quad  x\in [-\del, \del]^n.  
\end{align*}
Then
\[
  \left| \log |f (x) | \right|^p 
  \leq C_p \left( \left| \log \frac{1}{|u (x) |}  \right| ^p 
  + \left| \sum_{j=1}^m  \log \frac{1}{|x_n - \xi_j (x') | } \right|^p \right) ,  \quad x\in [-\del, \del]^n. 
\] 
Hence
\[ 
  \int_{[-\del, \del]^n }  \left| \log |f(x)| \right|^p \, dV(x) \lesssim \int_{[-\del, \del]^n } \left| \log \frac{1}{|u(x)|} \right|^p \, dV(x)
  + \int_{[-\del, \del]^n } \left|  \sum_{j=1}^m  \log \frac{1}{|x_n - \xi_j (x') | } \right|^p  \, dV(x).  
\] 
The first integral is bounded since $u$ does not vanish near $\0$. For the second integral, we apply Fubini's theorem to get 
\begin{align*}
    \int_{[-\del, \del]^n } \left| \sum_{j=1}^m \log \frac{1}{|x_n - \xi_j (x') | } \right|^p \, dV(x)
&= \int_{[-\del, \del]^{n-1} } \left( \int_{[-\del, \del]} \left| \sum_{j=1}^m \log \frac{1}{|x_n - \xi_j (x') | } \right|^p \, dx_n\right) \, d x'
\\ 
&\lesssim C_p \int_{[-\del, \del]^{n-1}} \left( \sum_{j=1}^m \int_{[-\del,\del]} \left| \log \frac{1}{|x_n - \xi_j (x') | } \right|^p \, dx_n \right) \, dx'. 
\end{align*}
By the condition on the coefficients $a_j(x')$, we know that
$|a_{m-2}(x')| = \left| \sum_{j \neq k} \xi_j(x') \xi_k (x') \right| $ and $a^2_{m-1}(x') = \left( \sum_{j=1}^n \xi_j (x') \right)^2 $ are small when $x'$ is close to $0$. By writing 
\[
  \sum_{j=1}^n (\xi_j(x'))^2 = \left( \sum_{j=1}^n \xi_j(x')  \right)^2 - 2 \sum_{j\neq k} \xi_j(x') \xi_k(x'),   
\]
we see that if $x' \in [-\del,\del]^{n-1}$, then $\sum_{j=1}^n |\xi_j(x')|^2 \leq a_{m-1}^2(x') + 2 |a_{m-2}(x')| \leq c_0$ for some $c_0>0$.           
Consequently, the inner integral is bounded by (up to a constant multiple independent of $x'$), 
\[ 
   \int_{ [-\del, \del] } \left( \log \frac{1}{|x_n| } \right)^p  \, d x_n. 
\] 
Since $\log \frac{1}{|x_n|} \leq C_{\ve} |x_n|^{-\ve} $, for any positive $\ve > 0$, we see that for each $0<p<\infty$, the above integral is bounded by some constant $C_p$ which is uniform in $x'\in [-\del,\del]$. This shows that $\log |f|$ is in  $L^p_{\loc}(\0)$ for any $0<p< \infty$. 
\end{proof} 
In what follows we fix the o-minimal structure $\R_{an}$. Given the germ of a real analytic function $f$ at the origin, we may assume without loss of generality that $f$ is real analytic in an open set $U$ containing the cube $[-1,1]^n$. We define the restricted function 
\eq{wti_f}
  \wti f(x) = \begin{cases} 
  f(x), \quad \text{for $x \in [-1,1]^n$};   
  \\ 1, \quad \text{otherwise.} 
  \end{cases} 
\eeq
\begin{lemma}
 $Z_{\wti f}$ is definable in $\R_{an}$. 
\end{lemma}
\begin{proof}
 We can write $Z_{\wti f}$ as the intersection of the graph of $\wti f$ 
 \[ 
 \{ (x_1,\dots,x_n,x_{n+1}) \in \R^{n+1}: \wti f(x_1,\dots,x_n) = x_{n+1} \}
 \] and the set $\{ x_{n+1}=0 \}$ in $\R^{n+1}$, both of which are definable in $\R_{an}$.   
\end{proof} 

\begin{prop}  \label{Prop::main_est} 
  Let $f$ be the germ of a real analytic function at the origin in $\R^n$, $n\geq 2$. Suppose $\codim_\0 Z_f \geq 2$. Then there exists a small neighborhood $U$ of $\0$ such that 
 \[ 
  \int_{U} \left|  \na \log |f(x)| \right| \, dV(x) < C(f, n), 
 \]
 where the constant $C(f,n)$ depends only on $f$ and the dimension $n$.  
\end{prop} 
\begin{proof}
In view of the above remark (\ref{wti_f}), it suffices to show that 
\[
 \int_{[-1,1]^n \sm Z_f} \left| \pa_{x_1} \log |f| \right| \, dV(x) =  \int_{[-1,1]^n} \left| \pa_{x_1} \log |f| \right| \, dV(x) \leq C(f,n).
\]  
Fix $x'=(x_2,\dots, x_n)$ and write 
\[
 \pa_{x_1} \log |f(x_1,x')| = \frac{\pa_{x_1} |f(x_1,x')|}{|f(x_1,x')|} = \frac{\sign (f) (\pa_{x_1} f)}{|f(x)|} 
 := g_{x'}(x_1).   
\]
By \rp{Prop::monotonicity}, $g_{x'}$ changes sign on $[-1,1]$ by a finite number of times $M$ which is uniform in $x'\in [-1,1]^{n-1}$. Now for each fixed $x'$, we break up the interval $[-1,1]$ into at most $M$ subintervals $[a_i(x'),b_i(x')]$, on each of which $g_{x'}$ has the same sign. Then 
\begin{align*}
 \int_{[-1,1]^n } \left|\pa_{x_1} \log |f| \right| \, dV(x) &= \int_{[-1,1]^{n-1}} \int_0^1 |\pa_{x_1} \log |f| | \, dx_1 \, dx' 
 = \sum_{i=1}^M \int_{[-1,1]^{n-1}} \int_{a_i(x')}^{b_i(x')} |\pa_{x_1} \log |f| | \, dx_1 \, dx'
 \\ &= \sum_{i=1}^M \int_{[-1,1]^{n-1}} \left| \int_{a_i(x')}^{b_i(x')} \pa_{x_1} \log |f| \, dx_1 \right| \, dx'
 \\ &= \sum_{i=1}^M \int_{[-1,1]^{n-1}} \left| \log |f(b_i(x'), x')| - \log |f(a_i(x'), x')| \right| \, dx' 
 \\ &\leq \sum_{i=1}^M \int_{[-1,1]^{n-1}} |\log |f(b_i(x'),x')|| + |\log |f(a_i(x'), x')| | \, dx' . 
\end{align*}
The integrand is bounded by a constant multiple of $|\log \inf_{ [-1,1]} |f(\cdot, x')|| + |\log \sup_{[-1,1]} |f(\cdot, x')|| $. We can assume that $\sup_{x \in [-1,1]^n} |f(x)| \leq 1 $, so 
\eq{x1_pd_int_est} 
  \int_{[-1,1]^n} \left|\pa_{x_1} \log |f| \right| \, dV(x) \lesssim 1+ \int_{[-1,1]^{n-1}} | \log \inf_{[-1,1]} |f (\cdot, x')| \, dx'. 
\eeq 

We now break up $[-1,1]^{n-1}$ into dyadic regions of the form
\[
 E_j := \{ x' \in [-1,1]^{n-1}: \inf_{[-1,1]} |f(\cdot,x')| \in (2^{-j-1}, 2^{-j}) \}, \quad j=0,1,2,\dots.   
\] 
We claim that $E_\infty = \{ x' \in [-1,1]^{n-1}: \inf_{[-1,1]} |f(\cdot,x')| = 0 \}$ has measure $0$ with respect to the Lebesgue measure of $\R^{n-1}$. Indeed, $E_\infty$ is the image of the zero set $Z_f$ under the projection map $\pi: \R^n \to \R^{n-1}$. Hence by \rp{Prop::dim_proj}, we have 
\[
  \dim(E_\infty) \leq \dim (Z_f) \leq n-2, 
\] 
which implies the claim. 
Now, if $y' \in E_j$, by definition there exists some $y_1 \in [-1,1]$ such that $|f(y_1,y')| \in [2^{-j-1}, 2^{-j}] $. Denote $y =(y_1,y')$ and let $y_\ast \in Z_f$ be such that $\dist(y, Z_f) = |y -y_\ast|$. Denote by $y_\ast'$ the last $n-1$ coordinates of $y_\ast$. Then
\[
  \dist(y,Z_f) =|y - y_\ast| 
  \geq |y' - y_\ast'| \geq \dist(y', E_\infty), 
\] 
where we used the fact that $y_\ast' \in E_\infty$. By \rp{Prop::L_dist_ineq}, there exists some $\all > 0$ such that  
\[
 2^{-j} \geq |f(y)| \geq c_f \blb \dist(y,Z_f) \brb^\all \geq c_f 
 \blb \dist(y', E_\infty) \brb^\all. 
\] 
Hence $\dist (y',E_\infty) \leq C_f (2^{-j})^{\frac{1}{\all}}$ for all $y' \in E_j$, which implies that $E_j$ lies in a $C_f (2^{-j})^{\frac{1}{\all}}$ neighborhood of $E_\infty$. By \rc{Cor::vol_tame_set}, $|E_j| \lesssim (2^{-j})^{\frac{\codim Z_f}{\all}} \leq (2^{-j})^{\frac{2}{\all}}$. It then follows from \re{x1_pd_int_est} that 
\begin{align*}
 \int_{[-1,1]^n} \left|\pa_{x_1} \log |f| \right| \, dV(x) &\lesssim 1+ \sum_{j=0}^\infty \int_{E_j} | \log 2^{-j}| \, dx' \lesssim 1+ \sum_{j=0}^\infty (2^{-j})^{\frac{2}{\all}} |\log 2^{-j}| < \infty. \qedhere
\end{align*} 
\end{proof} 
 
\nid \tit{Proof of \rt{Thm::W11_intro} and \rc{Cor::W11_glob_intro}}. 

Combining \rp{Prop::logLp}, \rp{Prop::main_est} and \rp{Prop::rem_sing_2}, we obtain \rt{Thm::W11_intro}. To prove \rc{Cor::W11_glob_intro}, note that the assumption implies that the dimension of $Z_f$ at each point is at least 2. We now cover $\ov U$ by finitely many balls. Furthermore, we can shrink the balls so that
 \begin{enumerate}
     \item  
   For any given ball $B(p,r)$ with center $p$ and $f(p) \neq 0$, we have that $\ov{B(p,r)}  \cap Z_f = \emptyset$; 
    \item 
   For any given ball $B(p,r)$ with center $p$ and $f(p) = 0$, $\log |f| \in W^{1,1}(B(p,r))$, 
 \end{enumerate} 
 where (2) holds because of \rt{Thm::W11_intro}. The conclusion then follows easily.

In order to prove \rt{Thm::blow-up_intro}, we 
first show that the integral blows up when the exponent is equal to the dimension of the ambient space.    
\begin{prop} \label{Prop::n_blup}  
 Let $f$ be the germ of a real analytic function at the origin in $\R^n$ with $f(\0)= 0$. Then for every (sufficiently small) neighborhood $\Uc$ of $\0$ 
\[ 
 \int_{\Uc} \left| \frac{ \na f }{f} \right|^n \, d V(x)= \infty. 
\] 
\end{prop}  
\begin{proof}
We prove by contradiction. Without loss of generality we can assume $f$ is real analytic in $\Uc = B_\ve(\0)$, a ball of radius $\ve$ centered at $\0$. Suppose 
\eq{qtfinite}  
   \int_{\Uc} \left| \frac{ \na f }{f} \right|^n \, dV(x) < \infty. 
\eeq
Define
\[
  A(x) = 
  \begin{cases} 
  \left| \frac{\na f (x)}{f (x)}  \right|,  & f (x) \neq 0;  \\
  0, & f (x)=0. 
  \end{cases} 
\]  
Then $A \in L^n (B_{\ve} (\0))$. Using spherical coordinate $ x = \rho  \om$, $\om \in S^{n-1}$, we have
\eq{Vdint}
  \int_{B_{\ve} (\0)} A^n(x) \, d V(x) = \int_{S^{n-1}} \int_0^{\ve}  A^n( \rho \om) \rho^{n-1} \, d \rho \, d \si (\om),   
\eeq
where we denote by $d \si (\om)$ the area element on the $n-1$ dimensional unit sphere $S^{n-1} \subset \R^n$.  
By \re{qtfinite}, the integral in \re{Vdint} is finite. Hence Fubini's theorem implies that, for almost all $\om_0 \in S^{n-1}$, 
\[ 
   \int_{0}^{\ve} A^n (\rho \om_0) \rho^{\rho -1} \, d \rho < \infty, 
\] 
or 
\eq{vrld} 
A (\rho \om_0) \rho^{\frac{n-1}{n}} \in L^n (0, \ve).    
\eeq 
Fix $\om_0 \in S^{n-1}$ such that \re{vrld} holds and also $f(\rho \om_0)$ is not identically $0$ for $\rho \in (0,\ve)$. Define $\var_{\om_0} (\rho) := f(\rho \om_0)$, with $\var_{\om_0} (0) = 0$. By choosing $\ve$ small, we have $\var_{\om_0}(\rho) \neq 0 $ if $\rho \in (0,\ve) \sm \{ 0 \}$ ($\var$ is a one variable real analytic function and thus have isolated zero at $0$.) Taking derivative, 
\[
  |\var'(\rho) | = \left| \na f (\rho \om_0) \cdot \om_0 \right|  \leq  \left| \na f (\rho \om_0) \right|.  
\]
By definition of $A$, we have for $\rho \neq 0$, 
\begin{align*} 
|\var'(\rho) | &\leq | \na f (\rho \om_0) |  \\
&= A (\rho \om_0) | f ( \rho \om_0) |   \\
&= A (\rho \om_0) |\var( \rho) |   \\
 &= A (\rho \om_0) \rho^{\frac{n-1}{n}}  |\var(\rho) |  \rho^{\frac{1-n}{n}}. 
\end{align*}
Hence
\[
   \rho^{\frac{n-1}{n}} \frac{|\var'( \rho) |}{|\var ( \rho) | } \leq  A (\rho \om_0)   \rho^{\frac{n-1}{n}}, \quad \rho \neq 0. 
\] 
Since $\var (\rho)$ is real analytic and is not identically $0$, it has finite order vanishing, i.e. there exists $N$ such that 
\begin{gather*}
  \var (\rho) = \rho^N + O(\rho^{N+1}); \\
  \var'(\rho) = N \rho^{N-1} + O(\rho^N), 
\end{gather*}
and $\frac{\var'(\rho)}{\var(\rho)} \approx \frac{1}{\rho}$ for $\rho$ small. 
It follows that
\begin{align*}
 A(\rho \om_0)  \rho^{\frac{n-1}{n}}  \geq \rho^{\frac{n-1}{n}} \left| \frac{\var'(\rho)}{\var(\rho)} \right| 
 \approx \rho^{\frac{n-1}{n}} \rho^{-1}
 = \rho^{-\frac{1}{n}} \notin L^n(0, \ve), 
\end{align*}
which contradicts with \re{vrld}.   
\end{proof}  

We are now ready to prove \rt{Thm::blow-up_intro}, which gives a refinement of the above result based on information about codimension of the zero set. 
\begin{thm} 
  Let $f$ be the germ of a real analytic function at the origin in $\R^n$ with $f(\0) = 0$. Suppose $\codim_\0 (Z_f) = n-d$. Then for every (sufficiently small) neighborhood $\Uc$ of $\0$, 
  \[
     \int_{\Uc} \left| \frac{\na f}{ f} \right|^{n-d} \, dV = \infty.   
  \] 
\end{thm}
\begin{proof}
By definition, for any $\ve >0$ there exists a smooth point $x_0$ of $Z_f$ in $\Uc $ such that $\dim_{x_0} (Z_f) = d$. Then one can find a coordinate system $(y',y'') \in \R^d \times \R^{n-d}$, with $y'=(y_1, \cdots, y_d)$, $y''=(y_{d+1}, \cdots y_n)$, in a small neighborhood $\mc{W} $ of $x_0$ such that $\mc{W} \subset \Uc$ and
\[
  Z_f \cap \mc{W} = \{ y_{d+1} = \cdots = y_n = 0 \}. 
\]  
Here, $y = ( y_1, \cdots, y_d )$ is a local coordinate chart of $Z_f $ in $\mc{W}$. Assume for the moment that $\int_{\mc{W}} \left| \frac{\na f}{f} \right|^{n-d}\, dV < \infty $. Let $f_{y'} (y'') = f (y', y'')$ be the restriction of $f$ onto the $y''$-plane.
For each fixed $y'$, the function $f_{y'}: U' \subset \R^{n-d} \to \R$ satisfies $f_{y'}(\0) = 0$. By Fubini's theorem,
\[ 
  \infty > \int_{\mc{W}} \left| \frac{\na_y f (y)}{f(y)} \right|^{n-d} \, dV(y) 
  =  \int \int \left| \frac{\na f (y', y'')}{f(y', y'')} \right|^{n-d} \, dy'' \, dy' 
  \geq \int \int \left| \frac{(\na f_{y'}) (y'') }{f_{y'} (y'')}\right|^{n-d} \, dy'' \, dy'.   
\] 
Now the integral on the right-hand side is infinity by \rp{Prop::n_blup}, which is contradiction. Therefore we conclude that $\int_{\mc{W}} \left| \frac{\na f}{f} \right|^{n-d} \, dV = \infty$ and
\[ 
    \int_{\Uc} \left| \frac{\na f}{f} \right|^{n-d} \, dV  > \int_{\mc{W}} \left| \frac{\na f}{f} \right|^{n-d} \, dV = \infty. \qedhere
\] 
\end{proof}

For functions of one variable, the differential inequality 
$|df/dx| \leq V|f|$ is very rigid, as shown by the following result. The proof is suggested to us by Yifei Pan. 
\begin{prop} \label{Prop::1-dim} 
  Let $f$ be a real-valued continuous function on $(-1,1)$. Suppose $f(0) = 0$, $f \in W^{1,p}(-1,1)$, $p\geq 1$, and $|df/dx| \leq V |f|$ for $V \in L^1(-1,1)$. Then $f \equiv 0$ on $(-1,1)$. 
\end{prop}
\begin{proof}
 By \cite[Theorem 8.2]{Br11}, 
\[
 f(x) = f(0) + \int_0^x f'(t) \, dt
 = \int_0^x f'(t) \, dt. 
\] 
By replacing $V$ with $V+1$, we can assume that $V \geq 1$ on $(-1,1)$. 
It follows that for $\ve>0$,  
\begin{align*} 
   \int_0^\ve V(x) |f(x)| \, dx 
   &\leq \int_0^\ve V(x) \left( \int_0^x |f'(t)| \, dt \right) \, dx  
\\ &\leq \left( \int_0^\ve V(x) \, dx  \right) \left( \int_0^\ve V(t) |f(t)| \, dt \right).  
\end{align*}
 If $f$ it not identically $0$ near the origin, then $\int_0^\ve V(x) |f(x)| \neq 0$. Hence 
 \[
   1 \leq \int_0^\ve V(x) \, dx.  
 \] 
 Since $V \in L^1(-1,1)$, the right-hand side goes to $0$ as $\ve \to 0$, which is a contradiction. Hence $f$ must be identically $0$.   
\end{proof}

The following example shows that the integrability exponents in all of our results are sharp. 

\ex{Ex::sharp} 
Let $f$ be a polynomial in $\R^n$, $n \geq 2$, of the form
 \[ 
    f(x) = x_1^{2 r_1} + x_2^{2 r_2} + \cdots + x_k^{2 r_k},  \quad k \geq 2, 
 \]   
 where $r_k$ are integers satisfying $1 \leq r_1 \leq r_2 \dots \leq r_k$. 
 Then for any bounded neighborhood $U$ of $\0$, the following holds
\eq{exam_int_cvg}   
   \int_U \left| \frac{\na f}{f}\right|^{\gm}  \, dV < \infty 
\eeq
if and only if 
\[
  \gm <  1 + \frac{r_1}{r_2} + \cdots + \frac{r_1}{r_k}. 
\] 
\eex
\begin{proof} 
We have, up to a nonzero constant factor,  	
\[
\left| \frac{ \na f (x) }{f (x)} \right| \approx \frac{\sqrt{|x_1|^{2 (2r_1- 1)} +  \dots + |x_n|^{2 (2 r_k- 1)} }}{|x_1|^{2 r_1} + \dots + |x_n|^{2 r_k}}.  
\]	
Set $y_1 = |x_1|^{r_1}, \dots, y_k = |x_k|^{r_k}$. Then $|y_1| = y_1^{\frac{1}{r_1}}, \dots, |y_k| = y_k^{\frac{1}{r_k}}$, and
\[
\left| \frac{ \na f (x) }{f (x)} \right| = \frac{ \sqrt{y_1^{4 - \frac{2}{r_1}} + \dots + y_k^{4- \frac{2}{r_k}} } }{y_1^2 + \dots + y_k^2}. 
\] 
Writing $y= (y',y'')$ where $y'= (y_1, \dots, y_k)$ and $y''=(y_{k+1}, \dots, y_n)$, the integral becomes
\eq{eheadcoc1} 
\int_U \left| \frac{\na f (x) }{f(x)} \right|^{\gm} \, dV(x) 
= \int_{U''} \int_{U'_{y''}} \frac{ \left( y_1^{4 - \frac{2}{r_1}} + \dots + y_k^{4- \frac{2}{r_k}} \right)^{\frac{\gm}{2}} } { \left( y_1^2 + \dots + y_k^2 \right)^{\gm} } \frac{y_1^{\frac{1}{r_1} -1} \dots y_k^{\frac{1}{r_k} -1} }{r_1 \dots r_k} \, dy' \, dy'', 
\eeq 
where we denote by $U''$ the projection of $U$ onto the $y''$ space and $U'_{y''} := \{ y' \in \R^n: (y',y'') \in U \}$.    
Using spherical coordinate $y = \rho  \om$ with $\rho \in (0, 1)$ and $\om = (\om_1, \dots, \om_k) \in S^{k-1}$, we get 
\begin{align*}
\frac{ \left( y_1^{4 - \frac{2}{r_1}} + \dots + y_k^{4- \frac{2}{r_k}} \right)^{\frac{\gm}{2}} } { \left( y_1^2 + \dots + y_k^2 \right)^{\gm} }
&= \frac{ \left( (\rho \om_1)^{4 - \frac{2}{r_1} } + \dots + (\rho \om_k)^{4 - \frac{2}{r_k} }  \right)^{\frac{\gm}{2}} }{\rho^{2 \gm}} \\ 
&= \frac{ \left( \rho^{4 - \frac{2}{r_1}} \left(  \om_1^{4- \frac{2}{r_1}}  + \rho^{\all_2} \om_2^{4- \frac{2}{r_2}}   \dots + \rho^{\all_k} \om_k^{4 -  \frac{2}{r_k}}  \right) \right)^{\frac{\gm}{2}} }{\rho^{2 \gm}}  \\ 
&= \rho^{- \frac{\gm}{r_1}} \left( \om_1^{4 - \frac{2}{r_1}} + \rho^{\all_2} \om_2^{4 - \frac{2}{r_2}}  \dots + \rho^{\all_k} \om_k^{4 - \frac{2}{r_k}} \right)^{\frac{\gm}{2}}, 
\end{align*}
where we set 
\[
\all_2 = \frac{2}{r_1} - \frac{2}{r_2}, \quad \dots, \quad \all_k = \frac{2}{r_1} - \frac{2}{r_k}.  
\]
Hence up to a positive constant factor, the inner integral in \re{eheadcoc1} is bounded by 
\eq{eheadcoc2} 
\int_{\rho =0}^{1}  \int_{S^{k-1}} \frac{\left( \om_1^{4 - \frac{2}{r_1}} + \rho^{\all_2} \om_2^{4 - \frac{2}{r_2}}  \dots + \rho^{\all_k} \om_k^{4 - \frac{2}{r_k}} \right)^{\frac{\gm}{2}}  }{ \om_1^{1- \frac{1}{r_1}} \dots \om_k^{1- \frac{1}{r_k}}  }   \,d\si (\om) \, \left( \rho^{- \frac{\gm}{r_1}}  \rho^{\frac{1}{r_1} -1} \dots \rho^{\frac{1}{r_k} -1} \rho^{k-1} \right) \, d \rho, 
\eeq 
where $d \si(\om)$ denotes the area element on $S^{k-1}$.  
The integral over the sphere is bounded up to a constant factor by 
\eq{sphere_int}
  \int_{S^{k-1}} \frac{\om_1^{2 \gm - \frac{\gm}{r_1}} + \rho^{\frac{\all_2 \gm}{2}} \om_2^{2 \gm-  \frac{\gm}{r_2}} + \dots + \rho^{\frac{\all_k \gm}{2}} \om_k^{ 2 \gm - \frac{\gm}{r_k} }  }{ \om_1^{1- \frac{1}{r_1}} \dots \om_k^{1- \frac{1}{r_k}} }  \, d \si(\om). 
\eeq
By using the standard parametrization 
\[
  (\om_1, \dots, \om_k) 
  = ( \cos \thh_1, \, \sin \thh_1 \cos \thh_2, \, \sin \thh_1 \sin \thh_2 \cos \thh_3, \, \dots, \left( \prod_{i=1}^{k-2} \sin \thh_i \right) \cos \thh_\ell, \, \prod_{i=1}^{k-1} \sin \thh_i), 
\] 
where $\thh_1, \dots, \thh_{k-2} \in [0, \pi]$ and $\thh_{k-1} \in [0, 2\pi]$, and the formula for surface area element 
\[ 
  d \si(v) = (\sin^{k-2} \thh_1) (\sin^{k-3} \thh_2) \cdots \sin \thh_{\ell-1} \, d \thh_1 \cdots d \thh_\ell,  
\] 
we see that the integral \re{sphere_int} is bounded for any $\gm>0$. Hence the integral \re{eheadcoc2} is bounded if and only if 
\begin{align*}
-1 &< - \frac{\gm}{r_1} + \frac{1}{r_1} - 1 + \dots + \frac{1}{r_k} - 1 + k - 1 \\ 
&= - \frac{\gm}{r_1} + \frac{1}{r_1} + \dots + \frac{1}{r_k} -1, 
\end{align*}
or 
\[ 
\gm < 1 + \frac{r_1}{r_2} + \dots + \frac{r_1}{r_k}.  \qedhere
\] 
\end{proof}
\begin{rem}
 By choosing suitable values for $r_1, \cdots, r_k$, we see that for every $\ve>0$, there exists a polynomial $f$ such that $|\na f| / f \notin L^{1+\ve}(\0)$. This shows that $L^1$ integrability of $|\na \log |f||$ is the best possible in \rt{Thm::W11_intro}. On the other hand, for $g(x) = x_1^2 + x_2^2 + \cdots + x_k^2$, by the above computation we have $|\na g| /g \in L^p_\loc(\0)$, for any $p <k$ with $k= \codim_{\0}(Z_g)$, but $|\na g| /g  \notin L^k_\loc(\0)$. This proves the optimality of the exponent in \rt{Thm::blow-up_intro}. 
\end{rem} 

\bibliographystyle{amsalpha}

\bibliography{reference}  

\end{document}